\documentclass[11pt, twoside, letterpaper]{amsart}

\usepackage{t1enc}
\usepackage{latexsym}
\usepackage{amssymb}
\usepackage{graphicx}
\usepackage{amsmath}
\usepackage{amsthm}
\usepackage{amsfonts}
\usepackage{mathtools}
\usepackage{comment}
\usepackage{mathrsfs}
\usepackage{marginfix}
\usepackage[all]{xy}
\usepackage[british]{babel}
\usepackage{url}
\usepackage{tikz-cd}

\usepackage[hypertexnames=false,
    pdftex,
	pdfpagemode=UseNone,
	breaklinks=true,
	extension=pdf,
	colorlinks=true,
	linkcolor=blue,
	citecolor=blue,
	urlcolor=blue,
]{hyperref}

\newtheorem{thm}{Theorem}[section]
\newtheorem{prop}[thm]{Proposition}
\newtheorem{lem}[thm]{Lemma}
\newtheorem{cor}[thm]{Corollary}

\theoremstyle{definition}
\newtheorem{example}[thm]{Example}
\newtheorem{rmk}[thm]{Remark}
\newtheorem{defn}[thm]{Definition}

\numberwithin{equation}{section}

\newcommand{\C}{\mathbb{C}}
\newcommand{\A}{\mathbb{A}}

\newcommand{\N}{\mathbb{N}}

\newcommand{\Q}{\mathbb{Q}}
\newcommand{\R}{\mathbb{R}}
\newcommand{\Z}{\mathbb{Z}}
\newcommand{\G}{\mathbb{G}}

\newcommand{\cA}{\mathcal{A}}

\newcommand{\cC}{\mathcal{C}}

\newcommand{\cE}{\mathcal{E}}
\newcommand{\cF}{\mathcal{F}}

\newcommand{\cH}{\mathcal{H}}
\newcommand{\cI}{\mathcal{I}}
\newcommand{\cO}{\mathcal{O}}
\newcommand{\cL}{\mathcal{L}}
\newcommand{\cM}{\mathcal{M}}
\newcommand{\cT}{\mathcal{T}}

\newcommand{\fM}{\mathfrak{M}}

\DeclareMathOperator{\Pic}{Pic}
\DeclareMathOperator{\Coker}{Coker}
\DeclareMathOperator{\Ker}{Ker}

\DeclareMathOperator{\Hom}{Hom}

\DeclareMathOperator{\Cohs}{\textit{Coh}}

\DeclareMathOperator{\rk}{rk}
\DeclareMathOperator{\im}{Im}

\DeclareMathOperator{\codim}{codim}

\DeclareMathOperator{\Spec}{Spec}
\DeclareMathOperator{\Quot}{Quot}

\DeclareMathOperator{\Coh}{Coh}

\DeclareMathOperator{\gr}{gr}

\DeclareMathOperator{\Db}{D^b}

\DeclareMathOperator{\ST}{\overline{ST}}
\DeclareMathOperator{\BG}{B\G}
\DeclareMathOperator{\Tors}{Tors}
\DeclareMathOperator{\GL}{GL}
\DeclareMathOperator{\Aut}{Aut}
\DeclareMathOperator{\Todd}{Todd}
\DeclareMathOperator{\ch}{ch}

\usepackage{xcolor}

\begin{document}

\title[Moduli spaces of slope-semistable sheaves]{Moduli spaces of slope-semistable sheaves with reflexive Seshadri graduations}
\author{Mihai Pavel,  Matei Toma}

\address{Institute of Mathematics of the Romanian Academy,
P.O. Box 1-764, 014700 Bucharest, Romania
}
\email{cpavel@imar.ro}

\address{ Universit\'e de Lorraine, CNRS, IECL, F-54000 Nancy, France
}
\email{Matei.Toma@univ-lorraine.fr}


\begin{abstract}
    We study the moduli stacks of slope-semistable torsion-free coherent sheaves that admit reflexive, respectively locally free, Seshadri graduations on a smooth projective variety. We show that they are open in the stack of coherent sheaves and that they admit good moduli spaces when the field characteristic is zero. In addition, in the locally free case we prove that the resulting moduli space is a quasi-projective scheme.
\end{abstract}
\maketitle
\section{Introduction}

Apart from its usefulness in studying geometrical properties of a given projective variety, e.g. \cite{Miyaoka, CampanaPaun}, the concept of slope-stability of torsion-free coherent sheaves is important in algebraic geometry for offering a restrictive condition under which the construction of moduli spaces of coherent sheaves on a fixed projective variety becomes possible. It was first used in the classification of locally free sheaves over algebraic curves \cite{mumford63,narasimhan65stable,seshadri67} and later extended to higher dimensional bases polarized by some ample divisor  class \cite{takemoto1972stable}. However it turned out that in order to obtain reasonable compactifications of the obtained moduli spaces one needed to allow at the boundary {\em $S$-equivalence classes} of torsion-free sheaves which are Gieseker-Maruyama-semistable, \cite{gieseker77, maruyama78moduli,simpson1994moduli}. One feature of the Gieseker-Maruyama compactifications is that they corepresent a suitable moduli functor. For the known constructions of compact  moduli spaces of slope-semistable torsion-free sheaves  this feature is lost, \cite{HuybrechtsLehn, greb16moduli, greb2021HYM}. Here we exhibit an open subfunctor of the moduli functor of  coherent sheaves, which is corepresented by a moduli space and which contains all slope-polystable reflexive sheaves. 

The problem of constructing moduli spaces of slope-polystable locally free sheaves was recently addressed by Buchdahl and Schumacher in a  complex geometrical setting. Specifically they constructed 
a coarse moduli space of slope-polystable locally free sheaves which enjoys a universal property over the category of weakly normal analytic spaces \cite[Theorem 1]{BuchdahlSchumacher3}. The techniques employed rely on the Kobayashi-Hitchin correspondence and on analytic Geometric Invariant Theory.  In the process the authors proved that small deformations of polystable locally free sheaves enjoy  a special property, namely that they have locally free Seshadri graduations, see Definition \ref{def:Seshadri}. 

The above property, which we will call property $(SLF)$, is the object of study of this paper in an algebraic geometrical context. We will also consider the weaker property $(SR)$ of slope-semistable torsion-free sheaves to admit a reflexive Seshadri graduation. Our main result asserts that these properties are open in flat families of coherent sheaves and that they give rise to corresponding moduli spaces which corepresent natural functors on the category of (noetherian) schemes. More precisely we show:

\begin{thm}[Theorem \ref{cor:MainThm}, Corollary \ref{cor:MainThm}, Theorem \ref{thm:SchemeStr}]
Let $(X,H)$ be a polarized smooth projective variety over an algebraically closed field $k$ of zero characteristic, and fix a numerical class $\gamma \in K_{num}(X)$. Then the substacks $\Cohs_{X,\gamma}^{(SLF)}$ and $\Cohs_{X,\gamma}^{(SR)}$  of slope-semistable torsion-free sheaves of class $\gamma$ on $X$ with the property $(SLF)$, respectively $(SR)$, are open in the stack $\Cohs_X$ of all coherent sheaves and admit good moduli spaces $M_{X,\gamma}^{(SLF)}$, respectively $M_{X,\gamma}^{(SR)}$, in the sense of Alper \cite{alper2013good}.

In addition, $M_{X,\gamma}^{(SLF)}$ is a quasi-projective scheme over $k$.
\end{thm}

The main technical result that we use in the proof is the recent existence theorem of Alper, Halpern-Leistner and Heinloth for good moduli spaces, \cite{AlperHLH}.

Here are some further properties of the moduli space $M_{X,\gamma}^{(SLF)}$ that we obtain. 

Firstly, it contains the moduli space of locally free stable sheaves as an open subscheme, which is also an open subscheme of the moduli space of Gieseker-Maruyama (GM) semistable sheaves. All these spaces are good moduli spaces of appropriate open substacks of $\Cohs_X$. Note however that in general they are mutually distinct, cf. Example \ref{ex:GM}. 

Secondly, $M_{X,\gamma}^{(SLF)}$ is not proper in general. In the surface case we compactify it in Theorem \ref{thm:BridgelandCompact} via a certain good moduli space of Bridgeland semistable objects, which was recently shown to be projective by Tajakka \cite{Tajakka}. In higher dimensions there is a natural open embedding 
\[
    M^{(SLF),wn}_{X,\gamma} \to M^{\mu ss}_{X,\gamma},
\]
where $M^{(SLF),wn}_{X,\gamma}$ is the weak normalization of $M^{(SLF)}_{X,\gamma}$, and $M^{\mu ss}_{X,\gamma}$ is the (weakly normal) moduli space of slope-semistable sheaves constructed in \cite{greb2017compact}. Note that $M^{\mu ss}_{X,\gamma}$ is not a good moduli space in general.

\subsection*{Acknowlegements: } {We thank the anonymous referee for their constructive comments, which helped us improve the presentation.} 

The authors acknowledge financial support from  IRN ECO-Maths. MP was also partly supported by the PNRR grant CF 44/14.11.2022 \textit{Cohomological Hall algebras
of smooth surfaces and applications}.

\section{Properties of sheaves with reflexive Seshadri filtrations}

We will denote by $X$ a smooth projective variety of dimension $n$ over an algebraically closed field $k$. Let $H$ be an integral ample class on $X$. 

For a torsion-free sheaf $E \in \Coh(X)$, we define the slope of $E$ (with respect to $H$) by
\[
    \mu(E) \coloneqq \frac{c_1(E) \cdot H^{n-1}}{\rk(E)}.
\]

\begin{defn}[Slope-semistability]
A sheaf $E \in \Coh(X)$ is \textit{slope-semistable} (resp. \textit{slope-stable}) if 
\begin{enumerate}
    \item $E$ is torsion-free, 
    \item for any subsheaf $F \subset E$ with $0 < \rk(F) < \rk(E)$ we have
    \[
        \mu(F) \le \mu(E) \quad (\text{resp. }<).
    \]
\end{enumerate}
\end{defn}

We also will say $\mu$-(semi)stable instead of slope-(semi)stable.

Any slope-semistable sheaf $E$ on $X$ is known to admit a \textit{Seshadri filtration} 
\[
    E_\bullet: \quad 0 = E_0 \subset E_1 \subset \ldots \subset E_m = E
\]
such that its factors $E_i/E_{i-1}$ are slope-stable with $\mu(E_i/E_{i-1}) = \mu(E)$ for all $i$, \cite[Theorem 1.6.7]{HuybrechtsLehn}. In general neither the filtration nor its associated graded module $\gr(E_\bullet) \coloneqq \oplus_i E_i/E_{i-1}$ are unique, see \cite[Example 3.1]{BTT2017}.
 
\begin{defn}\label{def:Seshadri}
    We say that a torsion-free coherent sheaf $E$ on $X$ has {\emph{property $(SR)$}}, resp. $(SLF)$, if $E$ is $\mu$-semistable and for some  Seshadri filtration of $E$, its corresponding graded module is reflexive, resp. locally free. 
\end{defn}

In this section and the next one all statements will concern coherent sheaves satisfying property $(SR)$ or property $(SLF)$, but the proofs will only treat the case of property $(SR)$, the same arguments applying  for property $(SLF)$.

\begin{rmk}\label{rmk:reflexive}
Consider a short exact sequence
\[
    0 \to E \to F \to G \to 0
\]
of coherent sheaves on $X$. By using \cite[Tag 00LX]{stacks-project} one can easily check:
\begin{enumerate}
    \item If $E$ and $G$ are reflexive, resp. locally free, then $F$ is reflexive, resp. locally free.
    \item If $F$ is reflexive and $G$ is torsion-free, then $E$ is reflexive.
\end{enumerate}
Note that if $E$ satisfies $(SR)$, resp. $(SLF)$, then $E$ is in particular reflexive, resp. locally free.
\end{rmk}

\begin{prop}\label{prop:EquivProp}
    Let $E$ be a slope-semistable torsion-free sheaf. Then $E$ satisfies $(SR)$, resp. $(SLF)$, if and only if any torsion-free quotient $F$ of $E$ of slope $\mu(F)=\mu(E)$ is reflexive, resp. locally free.
\end{prop}
\begin{proof}
    We first assume that $E$ satisfies $(SR)$, and let $q: E \to F$ be a torsion-free quotient with $\mu(F)=\mu(E)$ and $G$ its kernel. It is easy to check that $G$ and $F$ are $\mu$-semistable with the same slope. Then $G$ and $F$ admit Seshadri filtrations, say $G_\bullet: 0 = G_0 \subset G_1 \subset \ldots \subset G_l = G$, and respectively $F_\bullet: 0 = F_0 \subset F_1 \subset \ldots \subset F_s = F$. We set $m = l + s$ and 
    \[
    E_j = \begin{cases}
       G_j, &\quad \text{if } 0 \le j \le l\\
       q^{-1}(F_{j-l}), &\quad \text{if } l+1 \le j \le m.\\
     \end{cases}
    \]
Then $0 = E_0 \subset E_1 \subset \ldots \subset E_m = E$ is a Seshadri filtration of $E$. As $E$ satisfies $(SR)$ {and by construction $q$ induces an isomorphism between $E_{j+l}/E_{j+l-1}$ and $F_j/F_{j-1}$ for $1 \le j \le s$}, we deduce that the graded factors $F_j/F_{j-1}$ of $F_\bullet$ are reflexive, and so $F$ is reflexive by Remark \ref{rmk:reflexive}. 

The converse follows similarly. 
\end{proof}

\begin{prop}\label{prop:DirectSum}
Let $E$ and $F$ be torsion-free sheaves on $X$ with $\mu= \mu(E)=\mu(F)$ that satisfy $(SR)$, resp. $(SLF)$. Then any  extension of $F$ by $E$ also satisfies $(SR)$, resp. $(SLF)$.
\end{prop}
\begin{proof}
Consider an extension 
\[
    0 \to E \xrightarrow{\iota} G \xrightarrow{p} F \to 0,
\]
and let $G \to Q$ be a torsion-free quotient of slope $\mu$. By Proposition \ref{prop:EquivProp}, it is enough to show that $Q$ is reflexive. Consider the following commutative diagram
\[
\xymatrix{
     0 \ar[r] & E \ar[r]^\iota \ar[d] &  G \ar[r]^p \ar[d] &  F \ar[d] \ar[r] & 0\\  
     0 \ar[r] & Q_E \ar[r] &  Q \ar[r] &  Q_F  \ar[r] & 0 }
\]
where $\iota$ and $p$ are the canonical morphisms, and the vertical maps are surjective. As $E$ and $F$ satisfy $(SR)$, it follows by Proposition \ref{prop:EquivProp} that $Q_E$ and $Q_F$ are reflexive. Hence $Q$ is also reflexive.
\end{proof}

The following properties are now immediate.

\begin{prop}\label{prop:abelianCategory}
The full subcategory of $\Coh(X)$ whose objects are the zero sheaf and the torsion-free sheaves with the property $(SR)$, resp. $(SLF)$, and fixed slope equal to $\mu$ is abelian, noetherian, artinian and closed under extensions.
\end{prop} 

\begin{rmk}
Recall that in an abelian, noetherian and artinian category, any object admits a Jordan-H\"older (JH) filtration and the associated graded object is unique, \cite[Tag 0FCK]{stacks-project}. Therefore, when working with torsion-free sheaves enjoying property $(SR)$, resp. $(SLF)$, the notion of Seshadri filtration coincides with the notion of JH filtration in the corresponding categories. {This shows in particular the following:}
\end{rmk}

\begin{prop}\label{prop:equivDef}
If $E$ is a $\mu$-semistable sheaf on $X$ which admits some  Seshadri filtration whose corresponding graded module is reflexive, resp. locally free, then this property holds for any of its Seshadri filtrations. In this case all Seshadri graduations of $E$ are reflexive, resp. locally free, and isomorphic to one another. 
\end{prop}

\begin{cor}
    Any reflexive, resp. locally free, $\mu$-polystable sheaf on $X$ satisfies $(SR)$, resp. $(SLF)$.
\end{cor}

\begin{example}
    Let $X$ be an abelian surface, $x$ be a closed point on $X$ and $L$ be a non-trivial element in $\Pic^0(X)$. A standard computation using Serre's locally freeness criterion shows that the middle term $E$ of any non-trivial extension
$$0\to L\to E \to \cI_x\to0$$
is locally free, cf. \cite[Section I.5]{OSS}, \cite[Section I.2]{Toma_dissertation}. In this case $E$ is reflexive, it is slope semistable with respect to any ample polarization on $X$ but it is not reflexively Seshadri filtered.
\end{example}

Semistable sheaves in the sense of Gieseker and Maruyama are automatically slope-semistable and it is known that the converse is not true in general. The following example shows that slope-semistable sheaves need not be GM-semistable even when they are moreover supposed  to have property $(SLF)$. 
\begin{example}\label{ex:GM}
Let $(X,H)$ be a polarized $K3$ surface over an algebraically closed field of characteristic zero. Then its tangent sheaf $\cT_X$ is slope-stable of slope zero, see \cite[Proposition 9.4.5]{huybrechtsK3}, and so is $\cO_X$ too. Their direct sum $\cO_X\oplus \cT_X$ is slope-polystable, has property $(SLF)$, but is not GM-semistable, since the Hilbert polynomials $P_{\cO_X}(m)= \frac{m^2}{2}h^2+2$ and  $P_{\cT_X}(m)= m^2h^2-20$ are not proportional. Note that over complex non-algebraic $K3$ surfaces the same example works; in this case one defines an analogue of the Hilbert polynomial with respect to some K\"ahler class. 
\end{example}
Similar examples may be constructed in rank two over polarized $K3$ surfaces with more special properties.  

\section{Openness of the properties (SR) and (SLF)}

\begin{lem}\label{lem:torsion-freeQuotient} 
Let $R$ be a discrete valuation ring  over $k$ with quotient field $K$ and residue field $k$, let $E$ be a  coherent sheaf on $X\times \Spec(R)$ with slope $\mu$, flat over $R$ and such that $E_k$ has property $(SR)$, resp. $(SLF)$. Let further $F$ be a quotient of $E$ flat over $R$ such that $F_K$ is torsion-free and has slope $\mu$. Then $F_k$ is torsion-free and is therefore even reflexive, resp. locally free. 
\end{lem}
\begin{proof}
We have $\mu(F_k)=\mu(F_K)=\mu$  by flatness. We write $B^0:=\Tors (F_k)$ for the torsion of $F_k$ and we suppose by contradiction that $B^0\ne0$. Set $G^0:=F_k/B^0$, the torsion-free part of $F_k$.  Then $\mu(G^0)\le\mu(F_k)=\mu=\mu(E_k)$ on one hand, and on the other,   $\mu(E_k)\le \mu(G^0)$ by the semistability of $E_k$. Thus $\mu(G^0)=\mu(E_k)$ and so, $G^0$ must be reflexive  by Proposition \ref{prop:EquivProp} and  $\codim _X(B^0)\ge 2$. 

We will next apply to $F$ the first step of Langton's method aimed at producing a subsheaf $F'$ of $F$, flat over $R$, whith $F'_K=F_K$ and  $F'_k$  torsion-free, cf. \cite[Exercise 2.B.2]{HuybrechtsLehn}.  This goes as follows, cf. \cite[Proof of Theorem 2.B.1]{HuybrechtsLehn}.
 Put $F^1:=\Ker(F\to G^0)$.
We get two exact sequences
$$0\to B^0\to F_k\to G^0\to 0$$
and
$$0\to G^0\to F_k^1\to B^0\to 0.$$
{The second one is obtained using the inclusions $\pi F \subset F^1 \subset F$ and applying the Snake Lemma to the commutative diagram} 
\[
    \begin{tikzcd}
        0 \ar[r] & F^1 \ar[r,"\cdot \pi"] \ar[d] & F^1 \ar[r] \ar[d] & F^1_k \ar[d] \ar[r] & 0 \\
        0 \ar[r] & F \ar[r,"\cdot \pi"] & F \ar[r] & F_k \ar[r] & 0.
    \end{tikzcd}
\]
Set $B^1:=\Tors(F^1_k)$. Then we get an injective morphism 
$$G^0\to F_k^1/B^1$$
whose bidual
$$G^0\to (F_k^1/B^1)^{\vee\vee}$$
is an isomorphism by the normality property of reflexive sheaves
and factorizes through $F_k^1/B^1$. Thus $G^0\to F_k^1/B^1$ is an isomorphism and it follows that the exact sequence $0\to G^0\to F_k^1\to B^0\to 0$ splits and therefore
$F^1_k\cong G^0\oplus B^0$ and $B^1\cong B^0$. 

If we continue the process by setting 
$G^1:=F^1_k/B^1$, $F^2:=\Ker(F^1\to G^1)$, $B^2:=\Tors(F^2_k)$, and so on, we see as above that the sequences $(G^n)_{n\in\N}$ and $(B^n)_{n\in\N}$ remain stationary and the argument in \cite[Proof of Theorem 2.B.1]{HuybrechtsLehn} or in \cite{langton1975valuative} leads to a contradiction, see also \cite[Proposition 3.3]{toma2020criteria}. This proves the Lemma.
\end{proof}

\begin{rmk}\label{rmk:analyticLemma}
An analogue of Lemma \ref{lem:torsion-freeQuotient} holds in the complex analytic setup, following the proof of \cite[Theorem 3.1]{toma2020criteria}. That is, for a compact K\"ahler manifold $(X,\omega)$, a family $E$ flat over the unit disk $\Delta$ with fibers of slope $\mu$ on $X$ such that $E_0$ has property $(SR)$, resp. $(SLF)$, and a quotient of $E \to F$ flat over $\Delta$ such that $F_t$ is torsion-free of slope $\mu$ for all $t \in \Delta \setminus \{0\}$, one has that $F_0$ is torsion-free and therefore reflexive, resp. locally free. 
\end{rmk}

\begin{prop}\label{prop:openness}
    The properties $(SR)$ and $(SLF)$ are open in flat families of sheaves.
\end{prop}

\begin{proof}
Let $S$ be a scheme of finite type over $k$, and let $\cF$ be an $S$-flat family of sheaves on $X$. Since the property of being slope-semistable is open, cf. \cite[Theorem 1.7]{maruyama81on}, we may assume that all fibers $\cF_s$ over $S$ are slope-semistable with slope $\mu$. Consider the set
\begin{align*}
     H = \{ P(F') \in \Q[T] :{}& \cF_s \to F' \text{ is a torsion-free quotient for a } \\
      {}& \text{geometric point $s \in S$ such that $\mu(F') = \mu$
      }\}.
\end{align*}
By Grothendieck's Lemma (see \cite[Lemma 1.7.9]{HuybrechtsLehn}), the set $H$ is finite. Consider the relative Quot scheme $\varphi : Q :=\Quot(\cF,H) \to S$ of quotients $[\cF_s \to F']$ with $P(F') \in H$, and let $Z \subset Q$ be the
subset of non-reflexive quotients in $Q$. As the property of being reflexive is open in flat families, we see that $Z$ is closed in $Q$, and so $Z$ is proper over $S$. Therefore the scheme-theoretic image $\varphi(Z)$ is closed in $S$. 

Let $Z'$ be the closed subscheme of $Z$ corresponding to quotients of $\cF$ which are not torsion-free. By Proposition \ref{prop:EquivProp} the subset of $S$ parameterizing fibers not satisfying property $(SR)$ is exactly $\varphi(Z\setminus Z')$. Thus it suffices to show that $Z\setminus Z'$ is proper over $S$. To do this we will use the valuative criterion of properness and Lemma \ref{lem:torsion-freeQuotient}. Let $R$ be a DVR over $k$ with quotient field $K$, and consider a commutative diagram such as follows
\[
    \xymatrix{
        \Spec(K) \ar[r] \ar[d]^\iota & Z \setminus Z' \ar[d] \ar@{^{(}->}[r]  & Z \ar[ld]^\varphi\\
        \Spec(R) \ar[r] & S
    }
\]
where $\iota$ is given by the inclusion $R \subset K$. As $Z$ is proper over $S$, there exists a unique morphism $f: \Spec(R) \to Z$ such that the above diagram with $f$ added remains commutative. As $Z$ is a subset of the relative Quot scheme, $f$ corresponds to an $R$-flat family $\cF_R \to \cE$ of quotients with slope $\mu$. Let us assume by contradiction that $f$ does not factorize through $Z \setminus Z'$, which implies that $\cF_k$ satisfies $(SR)$. Furthermore we see that $\cE_K = \cE \otimes_R K$ is torsion-free, since $\cE_K$ corresponds to the morphism $\Spec(K) \to Z \setminus Z'$. Thus, by Lemma \ref{lem:torsion-freeQuotient}, we obtain that $\cE_k$ is reflexive, which gives a contradiction.
\end{proof}

\begin{rmk}
The Zariski-openness of the $(SR)$ and $(SLF)$ properties also holds in the K\"ahler analytic setup. 
\begin{proof}[Sketch of proof]
Let $(X,\omega)$ be a compact K\"ahler manifold, $S$ a complex space and $\cF$ an $S$-flat family of sheaves of slope $\mu$ on $X$. We consider similarly to the proof of Proposition \ref{prop:openness} the union $Q$ of the irreducible components of the relative Douady space over $S$ of quotients of $\cF$ such that each such irreducible component contains a torsion-free quotient of slope $\mu$. By \cite[Corollary 6.3]{toma2021boundedness} the natural map $\varphi: Q \to S$ is proper. Let $Z \subset Q$ be the closed subset of non-reflexive quotients, and $Z' \subset Z$ the closed subset of quotients with torsion. As before, it is enough to show that $\varphi(Z \setminus Z')$ is closed in $S$. 

Let $s \in S \setminus \varphi(Z \setminus Z')$ be a point. If $s \notin \varphi(Z)$ then $S \setminus \varphi(Z)$ is a Zariski-neighbourhood of $s$. Suppose now that $s$ is in $\varphi(Z)$. By restricting $S$ to a Zariski-open neighbourhood of $s$, we may assume that all irreducible components of $Z$ intersect the fiber $\varphi^{-1}(s)$ non-trivially. Now let $T \subset Z$ be an irreducible component of $Z$ and let $z$ be a point in $T \cap \varphi^{-1}(s)$. Since $s \notin \varphi(Z \setminus Z')$, we have $\varphi^{-1}(s) \cap Z = \varphi^{-1}(s) \cap Z'$. Let $T' \subset Z' \cap T$ be an irreducible component of $Z'$ passing through $z$. We will show that $T' = T$, which will imply that $Z' = Z$, thus finishing the proof.

Suppose by contradiction that $T' \ne T$.  Then there exists a holomorphic curve $f: \Delta \to T$ such that $f(0) = z$ and $f(\Delta \setminus \{0\}) \subset T \setminus T'$, where $\Delta$ denotes the unit disk in $\C$. This contradicts Remark \ref{rmk:analyticLemma}.
\end{proof}

\end{rmk}

\section{Existence of good moduli spaces}

{
Here and in the next section we assume that the characteristic of the base field $k$ is zero. {Consider the stack $\Cohs_X$ of flat families of coherent sheaves on $X$ \cite[Tag 08KA]{stacks-project}, and fix a numerical class $\gamma \in K_{num}(X)$ corresponding to a torsion-free sheaf on $X$. We denote by $\Cohs_{X,\gamma}^{(SR)},\Cohs_{X,\gamma}^{(SLF)} \subset \Cohs_X$ the substacks of slope-semistable torsion-free sheaves of class $\gamma$ on $X$ which satisfy property $(SR)$, resp. $(SLF)$}. In this section we shall first show that $\Cohs_{X,\gamma}^{(SR)}$ admits a separated good moduli space (in the sense of Alper \cite{alper2013good}), then as a consequence we will obtain the existence of a good moduli space for $\Cohs_{X,\gamma}^{(SLF)}$ as well. 

We will use the following criterion for the existence of good moduli spaces:

\begin{thm}[see Theorem A in \cite{AlperHLH}]\label{thm:AHLH}
Let $\fM$ be an algebraic stack of finite presentation, with affine automorphism groups and separated diagonal over a noetherian algebraic space $S$ of
characteristic 0. Then $\fM$ admits a separated good moduli space if and only if is $\fM$ is $\Theta$-reductive and $S$-complete.
\end{thm}

We note that $\Cohs_{X,\gamma}^{(SR)} \subset \Cohs_X$ is an open substack by Proposition \ref{prop:openness}. Thus $\Cohs_{X,\gamma}^{(SR)}$ is an algebraic stack locally of finite presentation and with affine diagonal over $k$, since $\Cohs_X$ is so by \cite[Tag 09DS,Tag 0DLY]{stacks-project}. In particular, $\Cohs_{X,\gamma}^{(SR)}$ has affine automorphism groups and separated diagonal. We also know that $\Cohs_{X,\gamma}^{(SR)}$ is quasi-compact, since the set of $\mu$-semistable sheaves of class $\gamma$ on $X$ is bounded (see \cite[Theorem 1.1]{simpson1994moduli}). According to Theorem \ref{thm:AHLH}, in order to show that $\Cohs_{X,\gamma}^{(SR)}$ admits a separated good moduli space, it is enough to check that $\Cohs_{X,\gamma}^{(SR)}$ is both $\Theta$-reductive and $S$-complete. We shall see this below, after we recall the definition of $\Theta$-reductivity and $S$-completeness (following \cite{AlperHLH}).

\subsection{$\Theta$-reductivity}

Let $\Theta = [\A^1/\G_m]$ be the quotient stack corresponding to the standard contracting action of the multiplicative group $\G_m$ on the affine line $\A^1$ over $k$.

\begin{defn}[see Definition 3.10 in \cite{AlperHLH}]
A noetherian algebraic stack $\fM$ is \emph{$\Theta$-reductive} over $k$ if for every $R$ DVR over $k$, any commutative diagram
\[
\xymatrix{ \Theta_R \setminus \{0\} \ar[r] \ar[d] &  \fM \ar[d] \\  \Theta_R \ar[r] \ar@{-->}[ru] &  \Spec k }
\]
of solid arrows can be uniquelly filled in, where $0 \in \Theta_R \coloneqq \Theta \times \Spec R$ is the unique closed point.
\end{defn}

Let $R$ be a DVR over $k$ with quotient field $K$. Then the closed point of $\Theta_R$ is defined by $0 = [0/\G_m] \cong \BG_m$, where $\BG_m = [\text{pt}/\G_m]$ is the classifying stack of $\G_m$ over $k$. Moreover, note that $\Theta_R \setminus \{0\}$ is described by the following union
\[
    \Theta_R \setminus \{0\} =  ([(\A^1 \setminus \{0 \})/\G_m] \times \Spec R) \cup_{\Spec K} \Theta_K,
\]
with $[(\A^1 \setminus \{0\})/\G_m] \times \Spec R \cong \Spec R$ and $\Theta_K$ open substacks of $\Theta_R$. Thus, a morphism $\Theta_R \setminus \{0\} \to \Cohs_{X,\gamma}^{(SR)}$ corresponds to a commutative diagram as follows:
\[
\xymatrix{ \Spec R \ar[r] &  \Cohs_{X,\gamma}^{(SR)}  \\  \Spec K \ar[r] \ar[u] &  \Theta_K \ar[u] }
\]
We can think of a morphism $f: \Theta_K \to \Cohs_{X,\gamma}^{(SR)}$ as an object $f(1) = E_K \in \Cohs_{X,\gamma}^{(SR)}(K)$ together with a filtration $E_K^\bullet$ such that $f(0) = \gr(E_K^\bullet) \in \Cohs_{X,\gamma}^{(SR)}(K)$ (see \cite[Section 2.3]{AlperHLH} for more details).

Putting together the above considerations, we obtain the following characterization of $\Theta$-reductivity (see also \cite[Proposition 3.6]{weissmann2023stacky}): $\Cohs_{X,\gamma}^{(SR)}$ is $\Theta$-reductive if and only if for 
\begin{enumerate}
    \item every DVR $R$ over $k$ with quotient field $K$ and residue field $k$,
    \item every family $E \in \Cohs_{X,\gamma}^{(SR)}(R)$, and
    \item every filtration $E_K^\bullet$ of $E_K$ with $\gr(E_K^\bullet) \in \Cohs_{X,\gamma}^{(SR)}(K)$,
\end{enumerate}
the (uniquely) induced filtration $E_k^\bullet$ of $E_k$ satisfies $\gr(E_k^\bullet) \in \Cohs_{X,\gamma}^{(SR)}(k)$. 
}

\begin{prop}\label{prop:ThetaReductive}
    The stack $\Cohs_{X,\gamma}^{(SR)}$ is $\Theta$-reductive.
\end{prop}
\begin{proof}
Let $R$ be a DVR over $k$ with quotient field $K$ and residue field $k$. Let $E \in \Cohs_{X,\gamma}^{(SR)}(R)$ be an $R$-flat family of sheaves {with slope $\mu := \mu(\gamma)$} and consider a filtration $E_K^\bullet$ of $E_K$ with $\gr(E_K^\bullet) \in \Cohs_{X,\gamma}^{(SR)}(K)$. This induces a unique filtration 
\[
E^\bullet:\quad  0 = E^0 \subset E^1 \subset \ldots \subset E^m = E
\]
with $R$-flat factors $E^j/E^{j-1}$ {(see for example \cite[Proposition 3.8]{AlperBBLT})}. By flatness the slope of each $(E^j/E^{j-1})_k$ is equal to $\mu$. Moreover we can deduce by Lemma \ref{lem:torsion-freeQuotient} that each $(E^j/E^{j-1})_k$ is reflexive. We conclude that $\gr(E_k^\bullet) \in \Cohs_{X,\gamma}^{(SR)}(k)$ (see Proposition \ref{prop:abelianCategory}).
\end{proof}

{
\subsection{S-completeness}
Let $R$ be a DVR over $k$ with quotient field $K$, residue field $k$, and uniformizer $\pi \in R$. Consider the quotient stack
\[
    \ST_R \coloneqq [\Spec(R[s,t]/(st-\pi))/\G_m],
\]
where $s, t$ have $\G_m$-weights $1, -1$ respectively. Then $\ST_R$ has a unique closed point $0$, where both $s$ and $t$ vanish. Note that $\ST_R|_{\{s = 0\}} \cong \Theta_k$ and same for $t = 0$. Therefore $0 = \ST_R|_{\{st = 0\}} \cong \Theta_k \cup_{\BG_{m}} \Theta_k$, where $\BG_m = [\textnormal{pt}/\G_m]$. On the other hand,
\[
    \ST_R \setminus \{ 0 \} \cong \Spec R \cup_{\Spec K} \Spec R.
\]

\begin{defn}[see Definition 3.38 in \cite{AlperHLH}]
A noetherian algebraic stack $\fM$ is \emph{$S$-complete} over $k$ if for every $R$ DVR over $k$, any commutative diagram
\[
\xymatrix{ \ST_R \setminus \{0\} \ar[r] \ar[d] &  \fM \ar[d] \\  \ST_R \ar[r] \ar@{-->}[ru] &  \Spec k }
\]
of solid arrows can be uniquelly filled in.  
\end{defn}

\begin{rmk}
By definition, a morphism $\ST_R \to \Cohs_X$ corresponds to a coherent sheaf $E$ on $\ST_R \times X$ flat over $\ST_R$. According to \cite[Corollary 7.14]{AlperHLH}, we also have the following explicit characterization: A quasi-coherent sheaf $E$ on $\ST_R \times X$ corresponds to a $\Z$-graded quasi-coherent sheaf $\oplus_{n \in \Z} E_n$ on $X_R$ together with a diagram
 \begin{equation}\label{eq:Sdiagram}
 \xymatrix{
\cdots\ar@/^1pc/[r]^{s}& E_{n-1}\ar@/^1pc/[r]^{s}\ar@/^1pc/[l]^{t}&E_n\ar@/^1pc/[r]^{s}\ar@/^1pc/[l]^{t} &E_{n+1}\ar@/^1pc/[r]^{s}\ar@/^1pc/[l]^{t} &\cdots \ar@/^1pc/[l]^{t}\\}
\end{equation}
such that $st=ts=\pi$, {which holds at each level}. Moreover,
\begin{enumerate}
    \item $E$ is coherent if and only if each $E_n$ is coherent, $s : E_{n-1} \to E_{n}$ is an isomorphism for $n \gg 0$ and $t : E_n \to E_{n-1}$ is an isomorphism for $n \ll 0$.
    \item $E$ is flat over $\ST_R$ if and only if $s$ and $t$ are injective {at each level}, and the induced map $s : E_{n-1}/tE_n \to E_n / tE_{n+1}$ is injective for all $n$. 
\end{enumerate}
\end{rmk}
}

In order to prove $S$-completeness for $\Cohs_{X,\gamma}^{(SR)}$, we will use elementary transformations as presented in \cite{langton1975valuative} of which we first recall the basic facts.

We denote by $\xi$ the generic point of $X_k$ and by $\Xi$ the generic point of $X_K$. We let $i: X_k \to X_R$ and $j: X_K \to X_R$ be the inclusion morphisms. Let $E_K$ be a torsion-free sheaf of rank $r$ on $X_K$, and let $E \subset j_* E_K$ be a torsion-free subsheaf on $X_R$ such that $j^* E = E_K$ and $i^* E$ is torsion-free on $X_k$. Then $E_\xi$ defines a free $\cO_{X_R,\xi}$-submodule of $(E_K)_{\Xi}$ of rank $r$. Conversely, by \cite[Proposition 6]{langton1975valuative} for any rank $r$ free $\cO_{X_R,\xi}$-submodule $M$ of $(E_K)_{\Xi}$ there is a unique torsion-free submodule $E \subset j_* E_K$ on $X_R$ such that $j^* E = E_K$, $E_\xi = M$ and $i^* E$ is torsion-free on $X_k$. We recall the equivalence relation introduced by Langton on such submodules $M$, which declares $M$ and $\pi^n M$ to be equivalent. Note that two equivalent modules give rise to isomorphic $\cO_{X_R}$-modules as above.

Recall also that Langton defines two equivalence classes $[M]$, $[M']$ to be \textit{adjacent} if there exist decompositions $M = N \oplus P$ and $M' = N + \pi M$. In this case, as $\cO_{X_R,\xi}$ is a principal ideal domain, we can find a basis $(e_1,\ldots,e_r)$ of $M$ over $\cO_{X_R,\xi}$ such that $(e_1,\ldots,e_s,\pi e_{s+1},\ldots,\pi e_r)$ is a basis of $M'$ for some $s$. The inclusion $M' \subset M$ induces an inclusion of associated $\cO_{X_R}$-modules $E' \subset E$. Let $\alpha : E'_k \to E_k$ denote its restriction to $X_k$. One gets two exact sequences
\begin{align}\label{eq:sesEdges}
    &0 \to \Ker(\alpha) \to E'_k \to \im(\alpha) \to 0,\\
    &0 \to \im(\alpha) \to E_k \to \Coker(\alpha) \to 0, \label{eq:sesEdgesII}
\end{align}
and it is shown in \cite[Proposition 7]{langton1975valuative} that $\Coker(\alpha)$ is torsion-free and that there is a natural isomorphism $\Coker(\alpha) \cong \Ker(\alpha)$. One says in this case that $E'$ is an \textit{elementary transformation} of $E$. In Langton's terminology and notation the passage from $[M]$ to $[M']$ is called an \textit{edge} and is denoted by $[M] - [M']$.

The next result is analogous to \cite[Proposition 4.4]{greb2020moduliKahler}. 

\begin{lem}\label{lem:adjacentModules}
Let $E_K$ be a torsion-free sheaf of rank $r$ on $X_K$, and let $E,E' \subset j_* E_K$ be torsion-free subsheaves on $X_R$ with $j^* E = j^* E' = E_K$ such that $i^* E$ and $i^* E'$ are torsion-free on $X_k$. Then we have:
\begin{enumerate}
    \item The sheaves $E$ and $E'$ are connected by  a finite chain of edges 
    \[
    [M_0] - [M_1], [M_1] - [M_2], \ldots, [M_{q-1}] - [M_{q}].
\]
\item If in addition $E_k$ and $E'_k$ satisfy property $(SR)$ and if for $0 \le n \le q$ the $E_n$ are the corresponding $\cO_{X_R}$-modules associated to $M_n$ by Langton's construction, then the $\cO_{X_k}$-modules $(E_n)_k$ satisfy $(SR)$.
\end{enumerate}

\end{lem}
\begin{proof}
Set $M = E_\xi$, $M' = E'_\xi$. As $\cO_{X_R,\xi}$ is a principal ideal domain, we can find a basis $(e_1,\ldots,e_r)$ of $M$ over $\cO_{X_R,\xi}$ such that $(\pi^{m_1}e_1,\ldots,\pi^{m_r}e_r)$ is a basis of $M'$ for suitable integers $m_1 \le m_2 \le \ldots \le m_r$. Replacing $E_2$ by some $\pi^m E_2$, we may further assume that $m_1 = 0$. 

For integers $0 \le n_1 \le n_2 \le \ldots \le n_r$, denote by $M[n_1,\ldots,n_r] \coloneqq (\pi^{n_1}e_1,\ldots,\pi^{n_r}e_r)$ the $\cO_{X_R,\xi}$-submodule of $(E_K)_\Xi$ generated by $\pi^{n_1}e_1,\ldots,\pi^{n_r}e_r$. It is immediate to see that any such $M[n_1,\ldots,n_r]$ is adjacent to 
\[
    M[n_1,\ldots,n_r]_{adj} \coloneqq M[n_1 - \min(1,n_1),\ldots,n_r - \min(1,n_r)].
\]
Applying this procedure starting with $M_0 \coloneqq M[m_1,\ldots,m_r]$, we obtain a finite sequence of adjacent modules connecting $M'$ and $M$ given by
\[
    M_0, M_1 \coloneqq (M_0)_{adj}, M_2 \coloneqq (M_1)_{adj},\ldots, M_q \coloneqq (M_{q-1})_{adj}
\]
such that $M_q = M[0,\ldots,0]$. Thus $[M']$ and $[M]$ are connected via the finite sequence of edges
\[
    [M_0] - [M_1], [M_1] - [M_2], \ldots, [M_{q-1}] - [M_{q}].
\]

For the second statement, let us denote by $\alpha_n : (E_n)_k \to (E_{n+1})_k$ the morphism induced by the inclusion $M_n \subset M_{n+1}$. Clearly the restriction of $\alpha_n$ to $\cO_{X_k,\xi}$ is given by tensoring $M_n \to M_{n+1}$ by $\cO_{X_R,\xi}/\pi \cO_{X_R,\xi}$. By our construction, we see that $\Ker(\alpha_0) = \Ker(\alpha_q \circ \ldots \circ \alpha_0)$ over the generic point $\xi$. As $(E_0)_k$ is torsion-free, it follows that $\Ker(\alpha_0) = \Ker(\alpha_q \circ \ldots \circ \alpha_0)$ over $X_k$. Now $(E_0)_k = E'_k$ and $(E_q)_k = E_k$ satisfy $(SR)$, hence $\Ker(\alpha_q \circ \ldots \circ \alpha_0)$ also satisfies $(SR)$ by Proposition \ref{prop:abelianCategory}. Thus $\Ker(\alpha_0)$ satisfies $(SR)$. Using the short exact sequences \eqref{eq:sesEdges}, \eqref{eq:sesEdgesII} for $\alpha_0$, the isomorphism $\Coker(\alpha_0) \cong \Ker(\alpha_0)$ and Proposition \ref{prop:abelianCategory} we obtain that $(E_1)_k$ satisfies $(SR)$. We conclude by induction on the length $q$ of our chain of edges.
\end{proof}

\begin{prop}\label{prop:Scomplete}
    The stack $\Cohs_{X,\gamma}^{(SR)}$ is $S$-complete.
\end{prop}
\begin{proof}
Consider a morphism $\ST_R \setminus \{0\} \to \Cohs_{X,\gamma}^{(SR)}$, where $R$ is a discrete valuation ring as in the above setup. This corresponds to two families $E, E' \in \Cohs_{X,\gamma}^{(SR)}(R)$ such that $j_* E_K = j_* E'_K$ on $X_K$. 

In what follows we use the same notation as in the proof of Lemma \ref{lem:adjacentModules}. That is, after replacing $E_2$ by some $\pi^m E_2$, we assume that $M \coloneqq E_\xi$ has a basis $(e_1,\ldots,e_r)$ over $\cO_{X_R,\xi}$ and $M' \coloneqq E'_\xi$ has a basis $(\pi^{m_1}e_1,\ldots,\pi^{m_r}e_r)$ with $0 = m_1 \le m_2 \le \ldots \le m_r$. According to Lemma \ref{lem:adjacentModules}, $M'$ and $M$ are connected via a finite sequence of edges 
\[
    [M_0] - [M_1], [M_1] - [M_2], \ldots, [M_{q-1}] - [M_{q}],
\]
with $M_0 = M[m_1,\ldots,m_r]$ and $M_q = M[0,\ldots,0]$.
For $0 \le n \le q$, denote by $E_n \subset j_* E_K$ the unique torsion-free subsheaf on $X_R$ such that $j^* E = E_K$, $(E_n)_\xi = M_n$ and $i^* E_n$ is torsion-free on $X_k$. We set $E_n = E'$ for $n < 0$ and $E_n = E$ for $n > q$.

Now consider the $\Z$-graded coherent sheaf $\oplus_{n \in \Z} E_n$ on $X_R$ together with maps $t_n : E_n \to E_{n-1}$ and $s_n: E_n \to E_{n+1}$ given by
\begin{equation}\label{eq:Sdiagram2}
      \xymatrix{
\cdots\ar@/^1pc/[r]^{\pi}& E_{-2}\ar@/^1pc/[r]^{\pi}\ar@/^1pc/[l]^{1} & E_{-1}\ar@/^1pc/[r]^{\pi}\ar@/^1pc/[l]^{1}&E_0\ar@/^1pc/[r]^{1}\ar@/^1pc/[l]^{1} &E_1\ar@/^1pc/[r]^{1}\ar@/^1pc/[l]^{\pi} &E_2\ar@/^1pc/[r]^{1}\ar@/^1pc/[l]^{\pi} &\cdots \ar@/^1pc/[l]^{\pi}\\}
\end{equation}
Let us show that the above defines an $\ST_R$-flat coherent sheaf $\cE$ on $\ST_R \times X$ corresponding to the extension of $\ST_R \setminus \{0\} \to \Cohs_{X,\gamma}^{(SR)}$ to $\ST_R \to \Cohs_{X,\gamma}^{(SR)}$. By construction, each $E_n$ is coherent, $s_n$ is an isomorphism for $n \gg 0$ and $t_n$ is an isomorphism for $n \ll 0$. We also see that $s_n$ and $t_n$ are injective for all $n$. To show that $\cE$ is $\ST_R$-flat, let us check that $s :  E_{n-1}/t_n E_n \to E_n / t_{n+1}E_{n+1}$ is injective for all $n$. This is clear by construction for $n \le 0$ and $n > q$. Take now $0 < n \le q$. Note that all $E_{n-1}/\pi E_n$ can be also viewed as $\cO_{X_k}$-modules. Then $s$ induces a morphism $M_{n-1}/\pi M_n \to M_n/\pi M_{n+1}$ of $\cO_{X_R,\xi}$-modules, which is in fact the localization of $s$ at the generic point $\xi$ of $X_k$. It is easy to check using the construction that the morphism $M_{n-1}/\pi M_n \to M_n/\pi M_{n+1}$ is injective, so in order to have injectivity of $s$ it remains to check that $E_{n-1}/\pi E_n$ is torsion-free on $X_k$. But tensoring the sequence
\[
    \pi E_{n} \to E_{n-1} \to E_{n-1}/\pi E_n \to 0
\]
by $\cO_{X_k}$ we see that $E_{n-1}/\pi E_n $ is the cokernel of $\alpha : (\pi  E_{n})_k \to (E_{n-1})_k$, which we know to be torsion-free since $\pi E_n$ and $E_{n-1}$ are related by an elementary transformation. Furthermore, by Lemma \ref{lem:adjacentModules} the sheaves $E_{n-1}/\pi E_n $ have property $(SR)$.

We thus obtain a diagram
\[
\xymatrix{
    \ST_R \setminus \{ 0\} \ar[r] \ar[d] &  \Cohs_{X,\gamma}^{(SR)} \ar[d]^\iota \\
     \ST_R \ar[r]^f &  \cC oh_X}
\]
where $\iota$ is the natural inclusion and $f$ corresponds to $\cE$.  According to \cite[Corollary 7.14]{alper2013good}, the restriction $f|_{ \{s = 0 \}}$, resp. $f|_{ \{t = 0 \}}$, corresponds to a filtration 
\begin{align*}
    F_\bullet:&\quad \cdots \xleftarrow{t} E_n/sE_{n-1} \xleftarrow{t} E_{n+1}/sE_n \xleftarrow{t} \cdots,\\
    \text{resp. } G_\bullet:&\quad \cdots \xrightarrow{s} E_{n-1}/tE_n \xrightarrow{s}  E_n/tE_{n+1} \xrightarrow{s} \cdots,
\end{align*}
such that $f(0) = \gr(F_\bullet) \cong \gr(G_\bullet) \in \cC oh_X(k)$. To show that $f$ factorizes through $\Cohs_{X,\gamma}^{(SR)}$, it remains to check that $\gr(G_\bullet) \in \Cohs_{X,\gamma}^{(SR)}(k)$. Since each member of the filtration has property $(SR)$ and since by Proposition \ref{prop:abelianCategory} we are dealing with an abelian category which is closed under extensions, we conclude that $\gr(G_\bullet)$ satisfies $(SR)$.
\end{proof}

\begin{thm}\label{thm:MainThm}
The stack $\Cohs_{X,\gamma}^{(SR)}$ admits a separated good moduli space $M^{(SR)}_{X,\gamma}$ of finite type over $k$. Furthermore, the geometric points of $M^{(SR)}_{X,\gamma}$ are the $S$-equivalence classes of coherent sheaves on $X$ with property $(SR)$ and class $\gamma$.
\end{thm}
\begin{proof}
The first statement follows by putting together Theorem \ref{thm:AHLH} with Proposition \ref{prop:ThetaReductive} and Proposition \ref{prop:Scomplete}. For the second part we note that two points $E$ and $E'$ in $\Cohs_{X,\gamma}^{(SR)}(k)$ are identified in $M^{(SR)}_{X,\gamma}$ if and only if $\overline{\{ E\}} \cap \overline{\{ E'\}} \ne \emptyset$ in $|\Cohs_{X,\gamma}^{(SR)}(k)|$, cf. \cite[Theorem 4.16 (iv)]{alper2013good}. Moreover, for any $E \in \Cohs_{X,\gamma}^{(SR)}(k)$, there is a unique closed point in $\overline{\{ E\}}$ which is given by the polystable graded module $\gr_{JH}(E)$, see \cite[Lemma 3.11]{AlperBBLT} for a proof in the curve case.
\end{proof}

\begin{cor}\label{cor:MainThm}
The stack $\Cohs_{X,\gamma}^{(SLF)}$ admits a separated good moduli space $M^{(SLF)}_{X,\gamma}$ which is an open subspace of $M^{(SR)}_{X,\gamma}$. Furthermore, the geometric points of $M^{(SLF)}_{X,\gamma}$ are the $S$-equivalence classes of coherent sheaves on $X$ with property $(SLF)$ and class $\gamma$.
\end{cor}
\begin{proof}
It suffices to remark that $\Cohs_{X,\gamma}^{(SLF)}$ is a saturated open substack of $\Cohs_{X,\gamma}^{(SR)}$ with respect to the map $\Cohs_{X,\gamma}^{(SR)} \to M^{(SR)}_{X,\gamma}$, cf. \cite[Remark 6.2]{alper2013good}. This is clear by the description of the geometric points of $M^{(SR)}_{X,\gamma}$.
\end{proof}

{
\begin{rmk}\label{rmk:UnivProp}
By \cite[Theorem 6.6]{alper2013good} the good moduli space $\Cohs_{X,\gamma}^{(SR)} \to M^{(SR)}_{X,\gamma}$ is universal for maps to algebraic spaces, i.e., every map $\Cohs_{X,\gamma}^{(SR)} \to Y$ to an algebraic space factors uniquely as
\[
\xymatrix{ \Cohs_{X,\gamma}^{(SR)} \ar[d] \ar[dr] &   \\   M^{(SR)}_{X,\gamma} \ar@{-->}[r] &  Y }
\]
The same holds true for $\Cohs_{X,\gamma}^{(SLF)} \to M^{(SLF)}_{X,\gamma}$.
\end{rmk}
}


\section{The good moduli space $M^{(SLF)}_{X,\gamma}$}\label{sect:ModuliSpaceLF}

We denote by $\cE$ the universal family on $\Cohs_{X,\gamma}^{(SLF)} \times X$ and consider the Donaldson morphism {(see \cite[Section 8.1]{HuybrechtsLehn})}
\[
    \lambda_\cE : K(X) \to \Pic(\Cohs_{X,\gamma}^{(SLF)}), \quad u \mapsto \det Rp_*(\cE \otimes q^*u),
\]
where $p$ and $q$ are the natural projections
\[
    \xymatrix{ & \Cohs_{X,\gamma}^{(SLF)} \times X \ar[dl]_p \ar[dr]^q & \\  \Cohs_{X,\gamma}^{(SLF)} &  & X }
\]
Let $\cL \coloneqq \lambda_\cE(u)$ be the line bundle in $\Pic(\Cohs_{X,\gamma}^{(SLF)})$ corresponding to  
\[
    u \coloneqq - \chi(\gamma \cdot h^n)h^{n-1} + \chi(\gamma \cdot h^{n-1})h^n,
\]
where $h = [\cO_H] \in K(X)$. Note that $\chi(\gamma \cdot u) = 0$, and as a consequence one can show that $\cL$ on $\Cohs_{X,\gamma}^{(SLF)}$ descends to $M^{(SLF)}_{X,\gamma}$, see below.

\begin{prop}\label{prop:LBdescent}
The line bundle $\cL = \lambda_\cE(u)$ on $\Cohs_{X,\gamma}^{(SLF)}$ descends to a line bundle $L$ on $M^{(SLF)}_{X,\gamma}$.
\end{prop}
\begin{proof}
By \cite[Theorem 10.3]{alper2013good} we know that $\cL$ descends to a line bundle on the good moduli space $M^{(SLF)}_{X,\gamma}$ if and only if it has trivial stabilizer action at the closed points of $\Cohs_{X,\gamma}^{(SLF)}$. So consider a closed point $t \in \Cohs_{X,\gamma}^{(SLF)}$ corresponding to some polystable sheaf
\[
    E = E_1^{\oplus r_1} \oplus \ldots \oplus E_m^{\oplus r_m}
\]
such that $E_1,\ldots,E_m$ are pairwise non-isomorphic stable locally free sheaves of slope $\mu = \mu(E)$. Then
\[
    \Aut(E) \cong \GL_{r_1} \times \cdots \times \GL_{r_m},
\]
and as shown in \cite[Section 4.3]{Tajakka} a group element $(g_1,\ldots,g_m) \in \Aut(E)$ acts on the fiber of $\lambda_{\cE}(u)$ at $t$ by multiplication by
\[
    \det(g_1)^{\chi(E_1 \cdot u)} \cdots \det(g_m)^{\chi(E_m \cdot u)}.
\]
Note that $\mu(E_i) = \mu(E)$ for each $i$, {which gives
\begin{align*}
    \frac{\int_X \ch(E_i)H^{n-1}\Todd_X}{\int_X \ch(E_i)H^{n}\Todd_X} ={}& \frac{\mu(E_i)}{\deg(X)} + \int_X H^{n-1}\Todd_X \\
    ={}& \frac{\mu(E)}{\deg(X)} + \int_X H^{n-1}\Todd_X \\
    ={}& \frac{\int_X \ch(E)H^{n-1}\Todd_X}{\int_X \ch(E)H^{n}\Todd_X}.
\end{align*}
Since $\ch(h) = \ch(\cO_X) - \ch(\cO_X(-H))$, we also get
\[
    \ch(h^n) = H^n \quad \text{and}\quad  \ch(h^{n-1}) = H^{n-1} - \frac{n-1}{2}H^n,
\]
Therefore by using the Hirzebruch-Riemann-Roch formula we obtain
\[
    \frac{\chi(E_i \cdot h^{n-1})}{\chi(E_i \cdot h^n)} = \frac{\chi(E \cdot h^{n-1})}{\chi(E \cdot h^n)},
\]
which further gives
\begin{align*}
     \chi(E_i \cdot u) ={}&- \chi(E_i \cdot h^{n-1})\chi(\gamma \cdot h^n) + \chi(E_i \cdot h^n)\chi(\gamma \cdot h^{n-1}) \\
     ={}&- \chi(E_i \cdot h^{n-1})\chi(E \cdot h^n) + \chi(E_i \cdot h^n)\chi(E \cdot h^{n-1}) \\
     ={}& 0.
\end{align*}
}
Thus the stabilizer action on $\cL$ is trivial at the closed point $t \in \Cohs_{X,\gamma}^{(SLF)}$.
\end{proof}

\begin{rmk}
The statement and the proof of Proposition \ref{prop:LBdescent} remain valid if one replaces the condition $(SLF)$ by the condition $(SR)$.
\end{rmk}

\begin{prop}\label{prop:SemiampleLB}
There is an integer $N > 0$ such that the line bundle $L^{\otimes N}$ is globally generated and separates points on $M^{(SLF)}_{X,\gamma}$.
\end{prop}
\begin{proof}
The proof follows the same strategy as in \cite[Section 3.3.2]{greb2017compact}. By \cite[Corollary 5.4]{langer2004semistable} we can find a smooth complete intersection $C = H_1 \cap \ldots \cap H_{n-1} \subset X$ with $H_1,\ldots,H_{n-1} \in |aH|$ of sufficiently large degree $a$, such that for any family $\cE \in \Cohs_{X,\gamma}^{(SLF)}(S)$ the restriction $\cE|_{S \times C}$ is an $S$-flat family of Gieseker-semistable sheaves of class $\gamma_C \coloneqq \gamma|_C$ on $C$. Here, the $(SLF)$ property is essential. Indeed, the fact that $\cE$ is an $S$-flat family of locally free (and not merely reflexive) sheaves implies that the restriction $\cE|_{S \times C}$ remains $S$-flat. In addition, the $(SLF)$ property ensures that the fibers of the restricted family $\cE|_{S \times C}$  are indeed semistable on $C$ (for sufficiently large $a$ and $C$ smooth), since the restriction to $C$ of the Seshadri graduation of each fiber $\cE_s$, $s\in S$, will be torsion free, a condition needed in order to apply \cite[Corollary 5.4]{langer2004semistable}. Thus we obtain a natural morphism of stacks
\[
    \iota: \Cohs_{X,\gamma}^{(SLF)} \to \Cohs_{C,\gamma_C}^{Gss},
\]
where $\Cohs_{C,\gamma_C}^{Gss}$ is the algebraic stack of Gieseker semistable sheaves of class $\gamma_C$ on $C$. Moreover, if $\cE_C$ is the universal family on $\Cohs_{C,\gamma_C}^{Gss} \times C$ and 
\[
    w \coloneqq - \chi(\gamma_C \cdot h|_C)[\cO_C] + \chi(\gamma_C)h|_C \in K(C),
\]
then there is an isomorphism
\[
    \iota^* \lambda_{\cE_C}(w) \cong \lambda_{\cE}(u)^{a^{n-1}},
\]
cf. \cite[Proposition 3.4]{greb2017compact}. 

Note that $\Cohs_{C,\gamma_C}^{Gss}$ also admits a good moduli space $M^{Gss}_{C,\gamma_C}$, which is the Gieseker moduli space of semistable sheaves of class $\gamma_C$ on $C$, and it is well known that $\lambda_{\cE_C}(w)$ descends to an ample line bundle $A$ on $M^{Gss}_{C,\gamma_C}$. Hence $\iota$ induces a natural map between the corresponding good moduli spaces
\[
    \varphi : M^{(SLF)}_{X,\gamma} \to M^{Gss}_{C,\gamma_C}
\]
such that $\varphi^* A \cong L^{a^{n-1}}$, showing that $L$ is semiample. Additionally, if $a$ is sufficiently large, then for any two closed points $E, E' \in M^{(SLF)}_{X,\gamma}$ there is an isomorphism
\[
    \Hom(E,E') \to \Hom(E|_C,E'|_C),
\]
so $\varphi$ is injective at the level of geometric points. Therefore $\varphi$ separates the points of $M^{(SLF)}_{X,\gamma}$.
\end{proof}

\begin{thm}\label{thm:SchemeStr}
The moduli space $M^{(SLF)}_{X,\gamma}$ is a quasi-projective scheme 
over $k$ and $L$ is ample on $M^{(SLF)}_{X,\gamma}$.
\end{thm}
\begin{proof}
By the previous result and its proof, there exists a quasi-finite morphism
\[
    \varphi : M^{(SLF)}_{X,\gamma} \to M^{Gss}_{C,\gamma_C}
\]
with $M^{Gss}_{C,\gamma_C}$ a projective scheme. It follows by \cite[Proposition 3.1]{OlssonStarrQuot} that $M^{(SLF)}_{X,\gamma}$ is a scheme of finite type over $k$. By Zariski's Main Theorem \cite[Tag 05K0]{stacks-project} there exists a factorization
\[
    \xymatrix{ M^{(SLF)}_{X,\gamma} \ar[rd]_ \varphi \ar[rr]_ j &  &  T \ar[ld]^\pi \\  &  M^{Gss}_{C,\gamma_C} &  }
\]
where $j$ is an open immersion and $\pi$ is finite. As $M^{Gss}_{C,\gamma_C}$ is projective, we obtain that $T$ is projective as well, hence $M^{(SLF)}_{X,\gamma}$ is quasi-projective.  Moreover $L$ is ample on $M^{(SLF)}_{X,\gamma}$, since $ L^{a^{n-1}}\cong\varphi^* A \cong j^*\pi^* A$, where $A$ is the ample line bundle on $M^{Gss}_{C,\gamma_C}$ defined in the proof of Proposition \ref{prop:SemiampleLB}; cf. \cite[Tag 01PR]{stacks-project}. 
\end{proof}

The following shows that we can also obtain $M^{(SLF)}_{X,\gamma}$ as a good GIT quotient, in the sense of Mumford \cite{mumfordGIT}. Since the family of $\mu$-semistable sheaves of class $\gamma$ is bounded, there is a sufficiently large integer $m$ such that any $E \in \Cohs_{X,\gamma}^{(SLF)}(k)$ is Castelnuovo-Mumford $m$-regular, and so we can write $E$ as a quotient
\[
    H^0(X,E(m)) \otimes \cO_X(-m) \to E \to 0.
\]
Set $V \coloneqq k^{\oplus P_H(\gamma,m)}$, and let $R_{X,\gamma}^{(SLF)}$ be the open subscheme of $\Quot(V \otimes \cO_X(-m),\gamma)$ containing all quotients $[q: V \otimes \cO_X(-m) \to E]$ such that
\begin{enumerate}
    \item $E$ satisfies property $(SLF)$,
    \item $q$ induces a linear isomorphism $V \to H^0(X,E(m))$.
\end{enumerate}
Then there is a natural $\GL(V)$ action on $R_{X,\gamma}^{(SLF)}$, and one can show that $\Cohs_{X,\gamma}^{(SLF)} \cong [R_{X,\gamma}^{(SLF)}/\GL(V)]$, cf. \cite[Section 2.5.1]{greb2020moduliKahler}. 

For any closed point $E\in\Cohs_{X,\gamma}^{(SLF)}$
there is a section $\sigma \in H^0(\Cohs_{X,\gamma}^{(SLF)}, \cL^{\otimes N})$ such that $\sigma(E) \neq 0$ and $(\Cohs_{X,\gamma}^{(SLF)})_\sigma$ is cohomologically affine, since $L$ is ample on $M_{X,\gamma}^{(SLF)}$, cf. \cite[Propositions 3.10(i), 4.7(i)]{alper2013good}. 
Thus all closed points of $\Cohs_{X,\gamma}^{(SLF)}$ are semistable with respect to $\cL$, cf. \cite[Definition 11.1]{alper2013good}.  Then by \cite[Section 13.5]{alper2013good} we get that $R_{X,\gamma}^{(SLF)} \to M_{X,\gamma}^{(SLF)}$ is the good GIT quotient. 
Thus we have 
\begin{prop}\label{prop:GITalgebraic}
The moduli space  $M_{X,\gamma}^{(SLF)}$ is the good GIT quotient of $R_{X,\gamma}^{(SLF)}$ by the $\GL(V)$ action described above.   
\end{prop}

\begin{prop}\label{prop:GITanalytic}
When $k=\C$, the analytification $M_{X,\gamma}^{(SLF),an}$ of $M_{X,\gamma}^{(SLF)}$ is an analytic Hilbert quotient for the analytification of $R_{X,\gamma}^{(SLF)}$ with the corresponding $\GL(V)$ action.   
\end{prop}
\begin{proof}
    This follows immediately by \cite[Proposition 2.5]{Mayrand}.
\end{proof}

\begin{rmk}\label{remark:BS} 
We consider the following contravariant functors on the category $(Sch/k)$ of schemes of finite type over $k$ with values in $(Sets)$:
\begin{align*}
\Sigma(S) ={}& \{ E \in \Cohs_{X,\gamma}^{(SLF)}(S)  \}/\sim, \\
    \Sigma'(S) ={}& \{ E \in \Cohs_{X,\gamma}^{(SLF)}(S) \}/\sim', \\
    \Sigma''(S) ={}& \{ E \in \Cohs_{X,\gamma}^{(SLF)}(S)\}/\sim'',
\end{align*}
where we declare $E, F\in \Cohs_{X,\gamma}^{(SLF)}(S)$ to be equivalent by $\sim$ if they are isomorphic, by $\sim'$ if there exists a line bundle $L$ on $S$ such that $E\sim F\otimes p^* L$, and respectively by $\sim''$ if there exist filtrations $E_\bullet: 0 = E_0 \subset E_1 \subset \ldots \subset E_m = E$, $F_\bullet: 0 = F_0 \subset F_1 \subset \ldots \subset F_l = F$ with $S$-flat graduations 
such that $\gr(E_\bullet),\gr(F_\bullet)$ are in $\Cohs_{X,\gamma}^{(SLF)}(S)$ and
$\gr(E_\bullet)\sim'\gr(F_\bullet)$. 
When $k=\C$ we define in a similar way the functors $\Sigma^{an}$, $\Sigma'^{an}$, $\Sigma''^{an}$ on the category $(An)$ of complex analytic spaces.

{By Proposition \ref{prop:GITalgebraic}, 
we obtain that $R_{X,\gamma}^{(SLF)} \to M_{X,\gamma}^{(SLF)}$ is in particular a categorical quotient. Hence by standard arguments, see \cite[Lemma 4.3.1]{HuybrechtsLehn}, we get that $\Sigma$ and $\Sigma'$ are corepresented by $M_{X,\gamma}^{(SLF)}$. Furthermore, following the proof of \cite[Theorem 6.10]{maruyama2016moduli}, one deduces that $\Sigma''$ is also corepresented by $M_{X,\gamma}^{(SLF)}$.} The same arguments apply in the analytic setup, showing that $M_{X,\gamma}^{(SLF),an}$ corepresents the functors $\Sigma^{an}$, $\Sigma'^{an}$, $\Sigma''^{an}$.

Moreover $M_{X,\gamma}^{(SLF)}$ is a {\em coarse moduli space} for $\Sigma''$ in the sense that $M_{X,\gamma}^{(SLF)}$ corepresents $\Sigma''$ and the natural map $\Sigma''(\Spec k)\to \Hom(\Spec k,M_{X,\gamma}^{(SLF)})$ is bijective. In the same way, when $k=\C$,  $M_{X,\gamma}^{(SLF), an}$ is a  coarse moduli space for $\Sigma''^{an}$ in the category of complex analytic spaces.
\end{rmk}

When $X$ is a surface, we can show furthemore that $M^{(SLF)}_{X,\gamma}$ embeds in a suitable moduli space $M_{X,\gamma}^\sigma$ of Bridgeland semistable objects, which was recently shown to be projective by Tajakka \cite{Tajakka}. {Note that in this case $M^{(SLF)}_{X,\gamma}$ coincides with $M^{(SR)}_{X,\gamma}$, since any reflexive sheaf is locally free on a surface \cite[Corollary 1.4]{hartshorne1980stable}.}

We recall the main facts we will need about $M_{X,\gamma}^\sigma$. For every real number $\beta$, one can define a torsion pair $(\cT_\beta,\cF_\beta)$ on $\Coh(X)$ by setting
\begin{align*}
    \cT_\beta ={}& \{ E \in \Coh(X) \mid \mu(Q) \le \beta  \text{ for any nonzero quotient $Q$ of }E \}, \\
    \cF_\beta ={}& \{ E \in \Coh(X) \mid \mu(F) > \beta \text{ for any nonzero subsheaf $F$ of }E \}.
\end{align*}
This defines a tilted heart $\Coh^\beta(X) \coloneqq \langle \cF_\beta[1],\cT_\beta \rangle$ which is the subcategory of $\Db(X)$ whose objects are given by exact triangles of the form 
\[
    F[1] \to E \to T
\]
with $F = \cH^{-1}(E) \in \cF_\beta$ and $T = \cH^0(E) \in \cT_\beta$. 

By \cite[Theorem 1.1]{Tajakka} there exists a Bridgeland stability condition $\sigma = (\cA,Z)$ on $X$, where $\cA = \Coh^{\beta_0}(X)[-1]$ for some $\beta_0 \in \R$, which satisfies the following properties, see loc. cit. for further details. 

\begin{enumerate}
    \item The objects of $\cA \subset \Db(X)$ are in fact given by exact triangles of the form
\[
    F \to E \to T[-1]
\]
with $T = \cH^1(E)$ a zero-dimensional sheaf on $X$ and $F = \cH^0(E)$ a $\mu$-semistable torsion-free sheaf of slope $\mu(F) = \beta_0$. 
    \item The $\sigma$-polystable objects of class $\gamma$ in $\cA \subset \Db(X)$ are of the form
    \[
        E = F \oplus \left( \bigoplus_i \cO_{p_i}^{n_i}[-1]\right),
    \]
    where $F$ is a $\mu$-polystable locally free sheaf and the $\cO_{p_i}$ are the structure sheaves of closed points $p_i \in X$.
    \item The moduli stack $\cM^\sigma_{X,\gamma}$ of $\sigma$-semistable objects of class $\gamma$ on $X$ is an algebraic stack of finite type over $k$ admitting a good moduli space $M^\sigma_{X,\gamma}$ which is a projective scheme. Moreover the geometric points of $M^\sigma_{X,\gamma}$ correspond to $\sigma$-polystable objects of class $\gamma$ in $\cA$ as above.
\end{enumerate}

\begin{thm}\label{thm:BridgelandCompact}
If $X$ is a surface, there exists an open embedding $M^{(SLF)}_{X,\gamma} \to M^\sigma_{X,\gamma}$. 
\end{thm}
\begin{proof}
There is a natural open embedding of algebraic stacks $$\Cohs^{\mu ss}_{X,\gamma} \to \cM^\sigma_{X,\gamma},$$
{where $\Cohs^{\mu ss}_{X,\gamma} \subset \Cohs_X$ is the open substack of $\mu$-semistable sheaves of class $\gamma$ on $X$.} Indeed, an $S$-family in $\cM^\sigma_{X,\gamma}(S)$ is by definition an $S$-perfect complex $\cE \in \Db(S \times X)$ such that for any closed point $s \in S$ the derived fiber $\cE|^{\textbf{L}}_{\{s\} \times X}$ is an $\sigma$-semistable object of class $\gamma$ in $\cA$, and by \cite[Lemma 2.1.4]{Lieblich} the locus $S^\circ$ of points $s \in S$ such that $\cE|^{\textbf{L}}_{\{s\} \times X}$ is concentrated only in degree $0$ is open. Therefore the derived restriction of $\cE$ to $S^\circ$ yields an $S^\circ$-flat family of coherent sheaves in $\Cohs^{\mu ss}_{X,\gamma}$.

Combining this with the open embedding $\Cohs^{(SLF)}_{X,\gamma} \subset \Cohs^{\mu ss}_{X,\gamma}$, we may view $\Cohs^{(SLF)}_{X,\gamma}$ as an open substack of $\cM^\sigma_{X,\gamma}$. By the above description of the geometric points of $M^\sigma_{X,\gamma}$, this open substack is furthermore saturated with respect to the good moduli space $\cM^\sigma_{X,\gamma} \to M^\sigma_{X,\gamma}$. Hence we obtain that the good moduli space $M^{(SLF)}_{X,\gamma}$ is an open subscheme of the projective scheme $M^\sigma_{X,\gamma}$, cf. \cite[Remark 6.2]{alper2013good}, thus proving our claim.
\end{proof}

{
\begin{rmk}\label{rmk:WNcompact}
In higher dimensions there is a natural open embedding 
\[
    M^{(SLF),wn}_{X,\gamma} \to M^{\mu ss}_{X,\gamma},
\]
where $M^{(SLF),wn}_{X,\gamma}$ is the weak normalization of $M^{(SLF)}_{X,\gamma}$, and $M^{\mu ss}_{X,\gamma}$ is the (weakly normal) moduli space of slope-semistable sheaves constructed in \cite{greb2017compact}. The proof, which we omit, follows by the methods in \cite[Theorem 5.10 and Proposition 5.12]{greb2017compact}.
\end{rmk}
}


\bibliography{mainbib.bib}

\newcommand{\etalchar}[1]{$^{#1}$}
\providecommand{\bysame}{\leavevmode\hbox to3em{\hrulefill}\thinspace}
\providecommand{\MR}{\relax\ifhmode\unskip\space\fi MR }
\providecommand{\MRhref}[2]{%
  \href{http://www.ams.org/mathscinet-getitem?mr=#1}{#2}
}
\providecommand{\href}[2]{#2}
\begin{thebibliography}{{Mum}63}

\bibitem[ABB{\etalchar{+}}22]{AlperBBLT}
Jarod Alper, Pieter Belmans, Daniel Bragg, Jason Liang, and Tuomas Tajakka, \emph{Projectivity of the moduli space of vector bundles on a curve}, Stacks project expository collection (SPEC), Cambridge: Cambridge University Press, 2022, pp.~90--125.

\bibitem[AHLH23]{AlperHLH}
Jarod Alper, Daniel Halpern-Leistner, and Jochen Heinloth, \emph{Existence of moduli spaces for algebraic stacks}, Invent. Math. \textbf{234} (2023), no.~3, 949--1038.

\bibitem[Alp13]{alper2013good}
Jarod Alper, \emph{Good moduli spaces for {Artin} stacks}, Ann. Inst. Fourier \textbf{63} (2013), no.~6, 2349--2402.

\bibitem[BS22]{BuchdahlSchumacher3}
Nicholas Buchdahl and Georg Schumacher, \emph{An analytic application of {Geometric} {Invariant} {Theory}. {II}: {Coarse} moduli spaces}, J. Geom. Phys. \textbf{175} (2022), 13.

\bibitem[BTT17]{BTT2017}
Nicholas Buchdahl, Andrei Teleman, and Matei Toma, \emph{A continuity theorem for families of sheaves on complex surfaces}, J. Topol. \textbf{10} (2017), no.~4, 995--1028.

\bibitem[CP19]{CampanaPaun}
Fr{\'e}d{\'e}ric Campana and Mihai P{\u{a}}un, \emph{Foliations with positive slopes and birational stability of orbifold cotangent bundles}, Publ. Math., Inst. Hautes {\'E}tud. Sci. \textbf{129} (2019), 1--49.

\bibitem[{Gie}77]{gieseker77}
D.~{Gieseker}, \emph{{On the moduli of vector bundles on an algebraic surface}}, {Ann. Math. (2)} \textbf{106} (1977), 45--60.

\bibitem[GRT16]{greb16moduli}
Daniel Greb, Julius Ross, and Matei Toma, \emph{Moduli of vector bundles on higher-dimensional base manifolds - construction and variation}, Int. J. Math. \textbf{27} (2016), no.~7, 27.

\bibitem[GSTW21]{greb2021HYM}
Daniel {Greb}, Benjamin {Sibley}, Matei {Toma}, and Richard {Wentworth}, \emph{{Complex algebraic compactifications of the moduli space of Hermitian Yang-Mills connections on a projective manifold}}, {Geom. Topol.} \textbf{25} (2021), no.~4, 1719--1818.

\bibitem[GT17]{greb2017compact}
Daniel {Greb} and Matei {Toma}, \emph{{Compact moduli spaces for slope-semistable sheaves}}, {Algebr. Geom.} \textbf{4} (2017), no.~1, 40--78.

\bibitem[GT20]{greb2020moduliKahler}
Daniel Greb and Matei Toma, \emph{Moduli spaces of sheaves that are semistable with respect to a {K{\"a}hler} polarisation}, J. {\'E}c. Polytech., Math. \textbf{7} (2020), 233--261.

\bibitem[{Har}80]{hartshorne1980stable}
Robin {Hartshorne}, \emph{{Stable reflexive sheaves}}, {Math. Ann.} \textbf{254} (1980), 121--176.

\bibitem[HL10]{HuybrechtsLehn}
Daniel {Huybrechts} and Manfred {Lehn}, \emph{{The geometry of moduli spaces of sheaves. 2nd ed}}, Cambridge University Press, 2010.

\bibitem[Huy16]{huybrechtsK3}
Daniel Huybrechts, \emph{Lectures on {{\(K\)}}3 surfaces}, Camb. Stud. Adv. Math., vol. 158, Cambridge: Cambridge University Press, 2016.

\bibitem[{Lan}75]{langton1975valuative}
Stacy~G. {Langton}, \emph{{Valuative criteria for families of vector bundles on algebraic varieties}}, {Ann. Math. (2)} \textbf{101} (1975), 88--110.

\bibitem[{Lan}04]{langer2004semistable}
Adrian {Langer}, \emph{{Semistable sheaves in positive characteristic}}, {Ann. Math. (2)} \textbf{159} (2004), no.~1, 251--276.

\bibitem[Lie06]{Lieblich}
Max Lieblich, \emph{Moduli of complexes on a proper morphism}, J. Algebr. Geom. \textbf{15} (2006), no.~1, 175--206.

\bibitem[Mar78]{maruyama78moduli}
Masaki Maruyama, \emph{Moduli of stable sheaves. {II}}, J. Math. Kyoto Univ. \textbf{18} (1978), 557--614.

\bibitem[Mar81]{maruyama81on}
\bysame, \emph{On boundedness of families of torsion free sheaves}, J. Math. Kyoto Univ. \textbf{21} (1981), 673--701.

\bibitem[Mar16]{maruyama2016moduli}
\bysame, \emph{Moduli spaces of stable sheaves on schemes: restriction theorems, boundedness and the {GIT} construction. {With} collaboration of {T}. {Abe} and {M}. {Inaba}}, MSJ Mem., vol.~33, Tokyo: Mathematical Society of Japan, 2016.

\bibitem[May19]{Mayrand}
Maxence Mayrand, \emph{Kempf-{Ness} type theorems and {Nahm} equations}, J. Geom. Phys. \textbf{136} (2019), 138--155.

\bibitem[MFK93]{mumfordGIT}
D.~{Mumford}, J.~{Fogarty}, and F.~{Kirwan}, \emph{{Geometric invariant theory. 3rd enl. ed}}, 3rd enl. ed. ed., vol.~34, Berlin: Springer-Verlag, 1993.

\bibitem[Miy87]{Miyaoka}
Yoichi Miyaoka, \emph{The {Chern} classes and {Kodaira} dimension of a minimal variety}, Algebraic geometry, {Proc}. {Symp}., {Sendai}/{Jap}. 1985, {Adv}. {Stud}. {Pure} {Math}. 10, 449-476 (1987)., 1987.

\bibitem[{Mum}63]{mumford63}
D.~{Mumford}, \emph{{Projective invariants of projective structures and applications}}, {Proc. Int. Congr. Math. 1962, 526-530}, 1963.

\bibitem[NS65]{narasimhan65stable}
M.~S. Narasimhan and C.~S. Seshadri, \emph{Stable and unitary vector bundles on a compact {Riemann} surface}, Ann. Math. (2) \textbf{82} (1965), 540--567.

\bibitem[OS03]{OlssonStarrQuot}
Martin Olsson and Jason Starr, \emph{Quot functors for {Deligne}-{Mumford} stacks}, Commun. Algebra \textbf{31} (2003), no.~8, 4069--4096.

\bibitem[OSS11]{OSS}
Christian Okonek, Michael Schneider, and Heinz Spindler, \emph{Vector bundles on complex projective spaces}, Modern Birkh\"auser Classics, Birkh\"auser/Springer Basel AG, Basel, 2011, corrected reprint of the 1988 edition, with an appendix by S. I. Gelfand.

\bibitem[{Ses}67]{seshadri67}
C.~S. {Seshadri}, \emph{{Space of unitary vector bundles on a compact Riemann surface}}, {Ann. Math. (2)} \textbf{85} (1967), 303--336.

\bibitem[{Sim}94]{simpson1994moduli}
Carlos~T. {Simpson}, \emph{{Moduli of representations of the fundamental group of a smooth projective variety. I}}, {Publ. Math., Inst. Hautes \'Etud. Sci.} \textbf{79} (1994), 47--129.

\bibitem[{Sta}20]{stacks-project}
The {Stacks project authors}, \emph{The stacks project}, \url{https://stacks.math.columbia.edu}, 2020.

\bibitem[Taj23]{Tajakka}
Tuomas Tajakka, \emph{Uhlenbeck compactification as a {Bridgeland} moduli space}, Int. Math. Res. Not. \textbf{2023} (2023), no.~6, 4952--4997.

\bibitem[{Tak}72]{takemoto1972stable}
Fumio {Takemoto}, \emph{{Stable vector bundles on algebraic surfaces}}, {Nagoya Math. J.} \textbf{47} (1972), 29--48.

\bibitem[Tom92]{Toma_dissertation}
Matei Toma, \emph{Holomorphe {Vektorb{\"u}ndel} auf nichtalgebraischen {Fl{\"a}chen}}, Bayreuth: Univ. Bayreuth, 1992, \url{https://dev-iecl.univ-lorraine.fr/wp-content/uploads/2020/12/these.pdf}.

\bibitem[Tom20]{toma2020criteria}
\bysame, \emph{Properness criteria for families of coherent analytic sheaves}, Algebr. Geom. \textbf{7} (2020), no.~4, 486--502.

\bibitem[Tom21]{toma2021boundedness}
\bysame, \emph{Bounded sets of sheaves on relative analytic spaces}, Ann. Henri Lebesgue \textbf{4} (2021), 1531--1563.

\bibitem[WZ25]{weissmann2023stacky}
Dario Wei{\ss}mann and Xucheng Zhang, \emph{A stacky approach to identifying the semistable locus of bundles}, Algebr. Geom. \textbf{12} (2025), no.~2, 262--298.

\end{thebibliography}
\bibliographystyle{amsalpha}

\medskip
\medskip
\begin{center}
\rule{0.4\textwidth}{0.4pt}
\end{center}
\medskip
\medskip

\end{document}